\documentclass[12pt,reqno]{amsart}
\usepackage{color,lscape,latexsym,amsmath,amsfonts,amssymb,amscd,eurosym}
\usepackage[latin1,ansinew]{inputenc}
\usepackage{hyperref}
\usepackage{setspace}
\usepackage{amsthm}
\usepackage{epsfig} 
\usepackage{color} 
\usepackage{multicol}
\usepackage[all]{xy} 
\usepackage{graphicx} 
\usepackage{amsmath}
\usepackage{bbm}
\usepackage{mathabx}
\usepackage[british,UKenglish,USenglish,english,american]{babel}%
\makeindex
\usepackage{graphics}
\usepackage{mathtools}

\newtheorem{theorem}{Theorem}
\newtheorem{cor}{Corollary}
\newtheorem{definition}{Definição}[section]
\DeclareMathOperator{\Iso}{Iso}

\newcommand{\bpr}{\begin{proposition}}
\newcommand{\epr}{\end{proposition}}
\newcommand{\bco}{\begin{corollary}}
\newcommand{\eco}{\end{corollary}}
\newcommand{\blm}{\begin{lemma}}
\newcommand{\elm}{\end{lemma}}
\newcommand{\bdf}{\begin{definition}}
\newcommand{\edf}{\end{definition}}
\newcommand{\bpm}{\begin{pmatrix}}
\newcommand{\epm}{\end{pmatrix}}
\newcommand{\beq}{\begin{equation}}
\newcommand{\eeq}{\end{equation}}
\newcommand{\brm}{\begin{remark}}
\newcommand{\erm}{\end{remark}}
\newcommand{\bcj}{\begin{conjecture}}
\newcommand{\ecj}{\end{conjecture}}
\newcommand{\bqu}{\begin{question}}
\newcommand{\equ}{\end{question}}
\newcommand{\bex}{\begin{exercise}}
\newcommand{\eex}{\end{excercise}}
\newcommand{\bit}{\begin{itemize}}
\newcommand{\eit}{\end{itemize}}
\newcommand{\ben}{\begin{enumerate}}
\newcommand{\een}{\end{enumerate}}

\newcommand\Quotient[2]{
        \mathchoice
            {
                \text{\raise1ex\hbox{\thinspace $#1$}\Big{/} \lower1ex\hbox{$#2$} \thinspace}%
            }
            {
                #1\,/\,#2
            }
            {
                #1\,/\,#2
            }
            {
                #1\,/\,#2
            }
    }

\newcommand\GIT[2]{
        \mathchoice
            {
                \text{\raise1ex\hbox{\thinspace $#1$}\Big{/}\!\!\!\!\Big{/} \lower1ex\hbox{$#2$} \thinspace}%
            }
            {
                #1\,/\,#2
            }
            {
                #1\,/\,#2
            }
            {
                #1\,/\,#2
            }
    }

\begin{document}

\thispagestyle{empty}

\title[Fundamental group of a compact manifold]{A remark on the fundamental group of a compact negatively curved manifold}

\author[Alcides de Carvalho Júnior]{Alcides de Carvalho Júnior $^{\dag}$} 
\thanks{$^{\dag}$ The author was supported by CNPq-Brazil.}

\maketitle


\thispagestyle{empty}

\begin{abstract}
\bigskip
In this short note we survey some results about the fundamental group of a compact negatively curved manifold. In particular, we review a theorem of Gusevskij, see \cite{zbMATH03907299}, it states that the fundamental group of a compact negatively curved manifold does not belong to $\mathcal{C},$  where $\mathcal{C}$ is the smallest class of groups that contains all amenable groups and is closed under free products and finite extensions. The class $\mathcal{C}$ is quite natural and was introduced for the first time in \cite{Thurston}.
\end{abstract}

\renewcommand{\thefootnote}{\fnsymbol{footnote}} 
\renewcommand{\thefootnote}{\arabic{footnote}}

\renewcommand{\thefootnote}{\fnsymbol{footnote}} 
\renewcommand{\thefootnote}{\arabic{footnote}}

\section{Introduction}
\bigskip

Yau and Schoen \cite{MR1333601} proposed to characterize the groups that appear as the fundamental group of some compact manifold with negative sectional curvature.
The fundamental group $\pi_1(M)$ of a compact negatively curved manifold has been well studied over the years, for example, Milnor \cite{Milnor} showed that one such group must have exponential growth, Preissmann \cite{Preissmann} and Beyers \cite{Byers} asserted that any of its abelian or solvable subgroup must be infinite cyclic. By the Hadamard theorem one knows that the manifold is a $K(\pi_1(M),1)$ space which imposes certain conditions on the group, e.g., the group must be torsion free. Eberlein \cite{MR0400250} also showed that these groups contain a nontrivial free subgroup. But not everything is known, for example, there exist a conjecture that, for a compact negatively curved $M$ of even dimension $n$, the Euler characteristic $\chi(M)$ is negative for $n \equiv 2\mod  4$ and $\chi(M)>0$ for $n \equiv 0\mod  4.$ 

Hirsch and Thurston \cite{Thurston} observed that if this conjecture is true, then the fundamental group of a compact negatively curved manifold does not belongs to the class $\mathcal{C}$, where $\mathcal{C}$ denotes the smallest class of groups which contains all amenable groups and which satisfies: if $G$ and $H$ are two groups in $\mathcal{C}$ then their free product, $G\ast H$, is also in $\mathcal{C}$, and if a group $G \in \mathcal{C}$ has finite index in $K$ then $K$ also belongs to $\mathcal{C}$. In particular, $G\in \mathcal{C}$ if $G$ is solvable, free or of subexponential growth. The authors also remarked that this can be proved directly.
Based on this, Chen proved in \cite{Chen} that if $\pi_1(M) \in \mathcal{C}$ then $\pi_1(M)$ must be a free group. Thereof, Gusevskij \cite{zbMATH03907299} proves that $\pi_1(M)$ can never belong to $\mathcal{C}$. In these notes we give a direct proof that the fundamental group of a compact negatively curved manifold never belongs to $\mathcal{C},$ by observing that $\pi_1(M)$ has no finite index free subgroups.

By Preissmman's Theorem we know that for a compact negatively curved $M,$ $\pi_1(M)$ does not belong to the class of all abelian groups, see \cite[Theorem 10]{Preissmann}. Beyers \cite{Byers} ensures that $\pi_1(M)$ does not belong to the class of all solvable groups, while by Avez's Theorem \ref{Avez}, $\pi_1(M)$ does not belong the class of all amenable groups. Note that, the class of all abelian groups is a proper class of all solvable groups and the class of all solvable groups is a proper class of all amenable groups, and that class $\mathcal{C}$ contains the class of all amenable groups, thus $\pi_1(M)\notin \mathcal{C}$ can be interpreted as a natural generalization of previous results. 

\section{Preliminaries}
In this section, we introduce some notions on geometric group theory, which arise when one looks at groups as metric spaces. 
The key notion in this regard is that of a quasi-isometry, namely, an equivalence relation among metric spaces that equates spaces which look the same on the large scale \ref{qi}. 
\begin{definition}\label{qi}
Let $X,Y$ be metric spaces. An aplication $f : X \to Y$ is an $(L,C)$-quasi-isometric embedding if $$L^{-1}d_X(x,x')- C\leq d_Y(f(x),f(x')) \leq Ld_X(x,x')+ C$$  for all $x, x' \in X.$
\end{definition}
An ($L,C)$-quasi-isometric embedding is called an $(L,C)$-quasi-isometry, if it
admits a quasi-inverse map $\widehat{f} : Y \to X$ which is also an $(L,C)$-quasi-isometric
embedding, that is, an $(L,C)$-quasi-isometry satisfying:
$$ d_X(  \widehat{f}f(x), x) \leq C,\quad d_Y (f \widehat{f}(y), y) \leq C$$
for all $x \in X, y \in Y.$
\begin{definition}
Two metric spaces $X, Y$ are {\it quasi-isometric} if there exists a quasi-isometry
$X \to Y$ .
\end{definition}
The main example of quasi-isometry, which partly justifies the interest in such
maps, is given by the following result which says that the volume growth in universal covers of compact Riemannian manifolds and growth of their fundamental groups increase at the same rate, it was proved first by A. Schwarz and, 13 years later, by Milnor \cite{Milnor}, in other words growth is quasi-isometric invariant.
It is well known that hyperbolicity in the sense of Gromov, as well as the number of ends, are both quasi-isometric invariants.

Given $M$ a compact, connected, Riemannian manifold, let $\tilde{M}$
be its universal covering endowed with the pull-back Riemannian metric, then the
fundamental group $\pi_1(M)$ acts isometrically on $\tilde{M}.$
\begin{theorem}[Milnor-Schwarz]
Let $M$ be a compact Riemannian manifold. Then the group $\pi_1(M)$ is finitely generated, and the metric space $\tilde{M}$ is quasi-isometric
to $\pi_1(M)$ with some word metric.
\end{theorem}

Milnor-Schwarz's theorem is more general than this and has other interesting applications, for example, if $G_1$ is a finite index subgroup of a finitely generated group $G$ then $G_1$ is also finitely generated; moreover the groups $G$ and $G_1$ are quasi-isometric. In particular, all these invariants are preserved by quasi-isometry.

The next result classifies all the subgroups of a hyperbolic group (the Cayley graph is a hyperbolic  metric space in the sense of Gromov), namely,
\begin{theorem}[Tits alternative for hyperbolic groups]
Let $\Gamma$ be a subgroup of a hyperbolic group. We have one of the following three cases:
\begin{enumerate}
\item $\Gamma$ is finite;
\item $\Gamma$ contains an infinite cyclic subgroup of finite index;
\item $\Gamma$ contains a non-abelian free subgroup.
\end{enumerate}
\end{theorem}
\begin{proof}
See Ghys and de la Harpe \cite{Ghys}, page 157.
\end{proof}
\section{The fundamental group of compact negatively curved manifold}
Throughout this section, we shall assume that our manifold is compact and with negative sectional curvature. In particular, it follows that the sectional curvature is bounded by negative constants. 
\begin{theorem}\cite[Theorem 1]{Chen}\label{thm1}

Let $M$ be a compact negatively curved manifold. Then, any amenable subgroup of $\pi_1(M)$ is cyclic.

\end{theorem}

\begin{proof}
 Since $M$ is compact, we can apply the Milnor-Schwarz lemma. Then the fundamental group is quasi-isometric to the universal cover of $M$, which implies it is a hyperbolic group. By the Tits alternative for hyperbolic groups the subgroup is either finite, virtually infinite cyclic, or contains a non-abelian free group. Since every amenable group does not contain a non-abelian free group and $\pi(M)$ is torsion free, we have that any amenable subgroup is virtually infinite cyclic group.
\end{proof}

It is a well known result; that the number of ends $e(G)$ of a finitely generated group $G$ is $0, 1, 2$ or $\infty$; moreover  $e(G) = 0$ if and only if $G$ is finite, while $e(G) = 2$ if and only if $G$ is a virtually infinite cyclic group. Let $H$ be a subgroup of $\pi(M);$ since $\pi(M)$ is torsion free, $e(H) \neq 0;$ if $H$ is amenable finitely generated then $e(H) \neq \infty;$ then $e(H)\in \{1,2\}$. Is it possible to prove that $e(H)$ is necessarily $2$? Note that the only case one still needs to prove is $e(H)\neq 1.$ Yau \cite{MR0283726} prove that any solvable subgroup of $\pi(M)$ is finitely generated. If one can prove this directly for amenable groups then the discussion above can be used to provide another proof of theorem \ref{thm1}.

\begin{cor}[Beyers and Preissmman] Every abelian or solvable subgroup of $\pi(M)$ is infinite cyclic.

\end{cor}
Preissmman \cite{Preissmann} also proved that $\pi_1(M)$ is not abelian. This result admits the following (quite hard) generalization, see \cite{MR0256305}. To be precise:
\begin{theorem}[Avez]\label{Avez}
Let $M$ be a connected compact Riemannian manifold with non-positive curvature everywhere, which is not a flat manifold; then the fundamental group of $M$ is not amenable. 

\end{theorem}
The idea used to prove this result was bounding from below the volume of geodesic balls centered at a point in the universal covering. In order to achieve this the author uses comparison techniques on an associated Riccati equation and Birkhoff's theorem (on space and time means) applied to the geodesic flow on $T_1M.$

Here we prove an intermediate version of this theorem:

\begin{theorem}

Let $M$ be a compact negatively curved manifold. Then $\pi_1(M)$ is not amenable. Moreover, $\pi_1(M)$ has no finite index amenable subgroup.

\end{theorem}
\begin{proof}
By Hadamard theorem the universal cover of $M$ has one end and thus $\pi_1(M)$ is a one ended group by Milnor-Schwarz. On the other hand, if $\pi_1(M)$ is amenable, by Tits alternative it has two ends and this is a contradiction.
\end{proof}
Observe that the previous result can also be obtained as an immediate consequence of Eberlein's theorem \cite{MR0400250} since amenable groups do not contain a nontrivial free subgroup.

For the proof of Theorem \ref{thm7} we need a theorem due to Swan \cite{Swan} which generalizes a theorem of Stallings \cite{Stallings}.
\begin{theorem}[Swan] 
A torsion free group with a free subgroup of finite index is a free group.
\end{theorem}
In particular, the fundamental group of $M$ (compact, negatively curved) has no finite index free subgroup. 

Note that if $\pi_1(M) \in \mathcal{C}$ then $\pi_1(M)$ can be obtained by the following operations:
\begin{itemize}
\item Contain a finite index amenable subgroup.
\item Contain a subgroup with a free product of amenable subgroups of finite index.  
\end{itemize} 

\begin{theorem}\label{thm7}
Let $M$ be a compact negatively curved manifold. Then $\pi_1(M)$ does not belong to the class $\mathcal{C}$.
\end{theorem}
\begin{proof} 
In fact, we know that $\pi_1(M)$ has one end. Now, if $\pi (M)$ contains a finite index amenable subgroup, then it must be virtually $\mathbb{Z}$ and therefore it has $2$ ends and this is a contradiction. On the other hand, suppose $\pi_1 (M)$ contains a subgroup with a free product of amenable subgroups of finite index. In this case, $\pi_1(M)$ has finite index free subgroup and therefore it has infinite ends, which is another contradiction. Thus, $\pi_1(M)$ does not belong to the class $\mathcal{C}$.

\end{proof}

We now explain in more details how Gusevskij obtains Theorem 7 above. Let $\tilde{M}^n$, $n\geq 2$, be a complete, simply connected, Riemannian manifold with sectional curvature $K$ satisfying the condition $K \leq k <0$.  The main result of \cite{zbMATH03907299} is the following. 
\begin{theorem}[Gusevskij]
 Let $G$ be a discrete geometrically finite subgroup of $\Iso(\tilde{M})$ the group of isometries of $\tilde{M}$ that does not contain parabolic elements. Then there exist a $G$-equivariant homeomorphism of the completion of $G$ onto the limit set of $G$.
 \end{theorem}

Some comments must be done regarding the last Theorem. For this, let $\tilde{M}(\infty)$ be the virtual boundary and consider the cone topology $\tau_{C}$ in $\bar{M}:=\tilde{M}\cup \tilde{M}(\infty).$ This topology was introduced by Eberlein and O'Neill in  \cite{MR0336648} and it is characterized by the following conditions: 1) the restriction of $\tau_C$ to $\tilde{M}$ coincides with the topology induced by the Riemannian distance $d$, 2) $\tilde{M}$ is an open everywhere dense subset of $\bar{M}$, 3) for any $p\in \tilde{M}$ and $x \in \tilde{M}(\infty)$ the family of {\it truncated cone} with vertex in $p$ and containing $x$ form a local basis of $\tau_C$ at $x.$ Here truncated cone a is given by $T(v,\varepsilon,r) = \{y \in \bar{M}:\sphericalangle_p(\gamma_v(\infty),y)< \varepsilon\}\setminus\{q\in\tilde{M}: d(p,q)\leq r\},$ where $\sphericalangle_p(\gamma_v(\infty),y)$ denotes the angle between the vectors $v \in T_p\tilde{M}$ and $\gamma'_{py}(0)$ ($\gamma_{py}$ is the unique geodesic joining $p$ and $y$).
 
 It is well known that the set $\bar{M}$ endowed with the cone topology $\tau_C$ is homeomorphic to a closed ball in $\mathbb{R}^n,$ and $\tilde{M}(\infty)$ is homeomorphic to the $(n-1)$-dimensional sphere $\mathbb{S}^{n-1}.$
 
 Let $G$ be a discrete subgroup of $\Iso(\tilde{M}).$  The set of limit poins in $\tau_C$ of the orbit $G\cdot p$ of an arbitrary point $p \in \tilde{M}$ is called {\it cone limit} set of $G$ at $p.$ This set does not depend on the choice of $p \in \tilde{M}$ and it is denoted by $L(G).$
 
 We denote by $Cay(G,S)$ the undirected Cayley graph of a finitely generated group $G$ with respect to a finite set of generators $S.$ Let $f:\mathbb{N}\to \mathbb{R}^{+}$ be a decreasing integrable function that satisfies the following condition: for any $k\in \mathbb{N},$ there are $m,n \in \mathbb{N}$ such that $mf(r)\leq f(kr)\leq n f(r)$ and $f(0)=f(1).$ On $Cay(G,S)$ we define a metric depending on $f$ as follows: the length of the edge in $Cay(G,S)$ with vertices $g,h\in G$ is equal to the minimum between $f(|g|)$ and $f(|h|)$ and the distance between any two vertices of $Cay(G,S)$ is equal to the minimum length among of the lengths of all open polygons joining these vertices, where $|\cdot|$ denote the word metric on $G$.
 
We denote by $\overline{Cay(G,S)}$ the Cauchy completion of $Cay(G,S)$, with the metric $d_f$. The {\it completion} of $G$ is the metric space $\overline{G(S,f)} = \overline{Cay(G,S)}\setminus Cay(G,S).$ For example, consider $G=\mathbb{Z}$ with generators $S=\{\pm 2,\pm 3\}$ and $f(n) = n^{-p}$, $p>1.$  The distance between the vertices $-3$ and $3$ is $3\cdot 2^{-p}.$ In this case, the completion of $\mathbb{Z}$ is a metric space formed by two points.

In order to prove Theorem 8, first of all, Gusevskij considers a special set $S_0$ of generators using that $G$ is a discrete geometrically finite subgroup of $\Iso(\tilde{M})$. Then he defines a map $\varphi: Cay(G,S_0)\to \tilde{M}$ and works to extend this map to $\bar{\varphi}:  \overline{Cay(G,S_0)}\to \tilde{M}\cup L(G).$ He concludes by showing that $\bar{\varphi},$ restricted to the completion of $ G $ gives a $G$-equivariant homeomorphism between $\overline{G(S_0,f)}$ and $L(G),$ for $f$ as in the example above.

Finally, the completion of a nontrivial free group is a totally disconnected metric space, while, by Theorem 8 the completion of the fundamental group of a compact Riemannian manifold $M^n$ of negative curvature is homeomorphic to $\mathbb{S}^{n-1}.$ Therefore, $\pi_1(M)$ is not free (see \cite[Theorem 3]{zbMATH03907299}.) By using this together with a result of Chen that states that if $\pi_1(M) \in \mathcal{C}$ then $\pi_1(M)$ must be a free group (see \cite[Theorem 3]{Chen}), Gusevskij obtains Theorem 7. 

\bigskip

\bibliographystyle{plain}
\bibliography{FirstLatexDoc_lib}

\vspace{1eM}
{\footnotesize
\begin{tabular}{lll}
	Alcides de Carvalho Júnior \\
	IMPA \\
	 Est. Dona Castorina 110 \\
	 Rio de Janeiro -- RJ, Brazil \\
	\textit{e-mail:} \texttt{alcidesj@impa.br} 
\end{tabular}}

\end{document}